\documentclass{amsart}

\usepackage{amsmath,amssymb}	

\usepackage{graphicx}		  

\usepackage{mathrsfs}
\usepackage{color}
\usepackage[all]{xy}%
\theoremstyle{plain}
    \newtheorem{theorem}{Theorem}[section]
   \newtheorem{proposition}[theorem]{Proposition}
    \newtheorem{lemma}[theorem]{Lemma}
    \newtheorem{corollary}[theorem]{Corollary}
    \newtheorem{conjecture}{Conjecture}
\theoremstyle{definition}
    \newtheorem{definition}[theorem]{Definition}
    \newtheorem*{notation}{Notation}
    \newtheorem{example}[theorem]{Example}
    \newtheorem{remark}[theorem]{Remark}
\newcommand{\bbR}{\mathbb{R}}

\DeclareMathOperator{\Gal}{Gal}

\DeclareMathOperator{\Hom}{Hom}

\DeclareMathOperator{\N}{N}

\DeclareMathOperator{\im}{im}

\newcommand{\bbC}{\mathbb{C}}

\newcommand{\id}{\mathrm{id}}

\newcommand{\bbZ}{\mathbb{Z}}

\newcommand{\bbQ}{\mathbb{Q}}

\newcommand{\ord}{\mathrm{ord}}

\newcommand{\sgn}{\mathrm{sgn}}

\title{On a conjecture for Rubin-Stark elements in a special case}
\author{Takamichi Sano}

\address{Department of Mathematics, 
Keio University, 
3-14-1, Hiyoshi, Kohoku-ku, Yokohama, 223-8522, Japan}
\email{tkmc310@a2.keio.jp}

\thanks{The author is supported by Grant-in-Aid for JSPS Fellows.}

\begin{document}

\maketitle

\begin{abstract}
We prove a conjecture on Rubin-Stark elements, which was recently proposed by the author, and also by Mazur and Rubin, in a special case. 
\end{abstract}

\section{Introduction}
In \cite[Conjecture 3]{sano}, motivated to generalize Gross's conjecture (\cite{G}) and Darmon's conjecture (\cite{D}), the author presented a conjecture concerning Rubin-Stark elements. After the author wrote the first version of \cite{sano}, Mazur and Rubin formulated in \cite[Conjecture 5.2]{MR} essentially the same conjecture as \cite[Conjecture 3]{sano}. In this paper, we prove this conjecture in a special case. 

We briefly recall the formulation of \cite[Conjecture 3]{sano}. Let $K/L/k$ be a tower of finite extensions of global fields, such that $K/k$ is abelian. Take $S$ and $T$, finite sets of finite places of $k$, satisfying certain conditions (see \S \ref{notation}). Take proper subsets $V \subset V' \subset S$ so that all $v \in V$ (resp. $V'$) split completely in $K$ (resp. $L$). Then, assuming the Rubin-Stark conjecture (\cite[Conjecture B$'$]{R2}), which predicts the existence of Rubin-Stark elements, our conjecture \cite[Conjecture 3]{sano} predicts the following equality: 
\begin{eqnarray}
\mathcal{N}_{K/L}(\varepsilon_{K,S,T,V})=\pm R_{V',V}(\varepsilon_{L,S,T,V'}), \label{conj}
\end{eqnarray}
where $\varepsilon_{K,S,T,V}$ and $\varepsilon_{L,S,T,V'}$ are Rubin-Stark elements for the data $(K/k,S,T,V)$ and $(L/k,S,T,V')$ respectively, $\mathcal{N}_{K/L}$ is the ``higher norm" introduced in \cite[Definition 2.12]{sano}, and $R_{V',V}$ is the ``algebraic regulator map", constructed by using the reciprocity maps at $v\in V'\setminus V$. 

In this paper, we prove the equality (\ref{conj}) under the following three assumptions:
\begin{itemize}
\item[(i)]{$V$ contains all infinite places of $k$,}
\item[(ii)]{all $v\in S$ split completely in $L$,}
\item[(iii)]{$\Gal(K/L)=\prod_{v\in S\setminus V}J_v$, where $J_v\subset \Gal(K/k)$ is the inertia group at $v$.}
\end{itemize}
(See Theorem \ref{mainthm}.) For example, the above assumptions are satisfied in the following case: $L$ is the Hilbert class field of $k$, $S$ is the union of all infinite places of $k$ and some principal prime ideals $\frak p_1,\ldots,\frak p_n$, $K$ is the composite field of the ray class fields modulo $\frak p_i^{e_i}$'s ($e_i$ is a positive integer), and $V$ is the set of all infinite places of $k$. 

Proving the main theorem, the author is inspired by the induction method used by Darmon in \cite[\S 8]{D}. By this method, Darmon proved a weaker statement of his conjecture, which he called ``order of vanishing" (see \cite[Theorem 4.2]{D}). We remark that Mazur and Rubin generalized this method directly to prove the ``order of vanishing" statement in a more general setting (see \cite[Theorem 6.3]{MR}). 
On the other hand, under our assumptions, we use Darmon's method to prove our conjecture completely. 

The organization of this paper is as follows. In \S \ref{ext}, we summarize useful constructions on exterior powers. In \S \ref{secRS}, we review the formulation of the Rubin-Stark conjecture, and summarize some known facts. In \S \ref{secconj}, we review the precise formulation of \cite[Conjecture 3]{sano}, and state the main theorem of this paper (Theorem \ref{mainthm}). In \S \ref{secpr}, we give the proof of the main theorem.

\begin{notation}

For any finite set $\Sigma$, the cardinality of $\Sigma$ is denoted by $|\Sigma|$.

For any abelian group $G$, $\bbZ[G]$-modules are simply called $G$-modules. The tensor product over $\bbZ[G]$ is denoted by $$-\otimes_G-.$$
Similarly, the exterior power over $\bbZ[G]$, and $\Hom$ of $\bbZ[G]$-modules are denoted by 
$$\bigwedge_G \ , \ \Hom_G(-,-),$$
respectively. 

For any subgroup $H$ of $G$, we define the norm element $\N_H\in \bbZ[G]$ by 
$$\N_H=\sum_{\sigma \in H}\sigma.$$




\end{notation}

\section{Exterior powers} \label{ext}

Let $G$ be a finite abelian group. For a $G$-module $M$ and $\varphi\in\Hom_G(M,\bbZ[G])$, there is a $G$-homomorphism
$$\bigwedge_G^rM \longrightarrow \bigwedge_G^{r-1}M$$
for all $r\in\bbZ_{\geq1}$, defined by 
$$m_1\wedge\cdots\wedge m_r \mapsto \sum_{i=1}^r (-1)^{i-1}\varphi(m_i)m_1\wedge\cdots\wedge m_{i-1}\wedge m_{i+1}\wedge\cdots \wedge m_r.$$
This homomorphism is also denoted by $\varphi$.

This construction gives a homomorphism
\begin{eqnarray}
\bigwedge_G^s\Hom_G(M,\bbZ[G]) \longrightarrow \Hom_G(\bigwedge_G^rM, \bigwedge_G^{r-s}M) \label{extmap}
\end{eqnarray}
for all $r, s\in \bbZ_{\geq0}$ such that $r\geq s$, defined by
$$\varphi_1\wedge\cdots\wedge \varphi_s \mapsto (m \mapsto \varphi_s \circ \cdots \circ \varphi_1(m)).$$
From this, we often regard an element of $\bigwedge_G^s\Hom_G(M,\bbZ[G])$ as an element of $\Hom_G(\bigwedge_G^rM, \bigwedge_G^{r-s}M)$.
Note that if $r=s$, $\varphi_1\wedge\cdots\wedge \varphi_r\in\bigwedge_G^r\Hom_G(M,\bbZ[G])$ and $m_1\wedge \cdots \wedge m_r \in \bigwedge_G^rM$, then we have 
$$(\varphi_1\wedge \cdots \wedge \varphi_r)(m_1\wedge \cdots \wedge m_r)= \det(\varphi_i(m_j))_{1\leq i, j \leq r}.$$

\section{The Rubin-Stark conjecture} \label{secRS}
In this section, we review the formulation of the Rubin-Stark conjecture (\cite[Conjecture B]{R2}). In \S \ref{notation}, we set notation which we use throughout this paper. In \S \ref{stateRS}, we state the Rubin-Stark conjecture. In \S \ref{properties}, we summarize some known properties of Rubin-Stark elements. 
\subsection{Notation} \label{notation} 
Let $k$ be a global field. We fix a separable closure $k^{\rm sep}$ of $k$, and any separable extension of $k$ is considered to be in $k^{\rm sep}$. We denote the set of all infinite places of $k$ by $S_\infty(k)$. For any finite separable extension $K/k$ and any set $\Sigma$ of places of $k$, we denote the set of places of $K$ lying above places in $\Sigma$ by $\Sigma_K$. Let $S$ and $T$ be finite sets of places of $k$. In this paper, we call the (ordered) pair $(S,T)$ {\it admissible} for the extension $K/k$ if the following conditions are satisfied:
\begin{itemize}
\item{$S$ is nonempty and contains $S_\infty(k)$ and all places ramifying in $K$,}
\item{$S\cap T=\emptyset$,}
\item{$\mathcal{O}_{K,S,T}^\times$ is torsion-free,}
\end{itemize}
where $\mathcal{O}_{K,S,T}^\times$ is the $(S,T)$-unit group of $K$, defined by 
\begin{eqnarray}
\mathcal{O}_{K,S,T}^\times :=\{ a\in K^\times \ | \ \ord_w(a)=0 \mbox{ for all } w\notin S_K \mbox{ and } a \equiv 1 \ (\text{mod } w') \mbox{ for all } w'\in T_K \}, \nonumber
\end{eqnarray}
where $\ord_w$ is the (normalized) additive valuation at $w$. 

Let $\Omega(k)$ be the set of quadruples $(K,S,T,V)$ satisfying the following:
\begin{itemize}
\item{$K$ is a finite abelian extension of $k$,}
\item{$S$ and $T$ are finite sets of places of $k$ such that $(S,T)$ is admissible for $K/k$,}
\item{$V$ is a proper subset of $S$ such that all $v\in V$ split completely in $K$.}
\end{itemize}
If we fix a finite set $T$ of finite places of $k$, then we define 
$$\Omega(k,T):=\{ (K,S,V) \ | \ (K,S,T,V) \in \Omega (k) \}.$$

Take $(K,S,T,V)\in \Omega(k)$. Let $\mathcal{G}_K$ denote the Galois group $\Gal(K/k)$. For a character $\chi \in \widehat{\mathcal{G}_K}:=\Hom(\mathcal{G}_K,\bbC^\times)$, the $(S,T)$-$L$-function is defined by 
$$L_{k,S,T}(s,\chi):=\prod_{v\in T}(1-\chi({\rm Fr}_v)\N v^{1-s}) \prod_{v \notin S}(1-\chi({\rm Fr}_v)\N v^{-s})^{-1},$$
where ${\rm Fr}_v \in \mathcal{G}_K$ is the Frobenius automorphism at $v$, and $\N v$ is the cardinality of the residue field at $v$. The product in the right hand side converges if ${\rm Re}(s)>1$. It is well-known that $L_{k,S,T}(s,\chi)$ has analytic continuation on the whole complex plane, and is holomorphic at $s=0$. We define 
$r_\chi=r_{\chi,S}:=\ord_{s=0}L_{k,S,T}(s,\chi).$ It is well-known that 
$$r_\chi=
\begin{cases}
|\{ v\in S \ | \ \chi(G_v)=1 \}| &\text{if $\chi\neq 1$}, \\
|S|-1 &\text{if $\chi=1$},
\end{cases}$$
where $G_v \subset \mathcal{G}_K$ is the decomposition group at $v$ (see \cite[Proposition 3.4, Chpt. I]{T}). Note that $r_\chi=r_{\chi^{-1}}$ for any $\chi \in \widehat{\mathcal{G}_K}$. For $r \in \bbZ_{\geq 0}$, define ``$r$-th order Stickelberger element" by 
$$\theta_{K/k,S,T}^{(r)}:=\sum_{\chi\in \widehat{\mathcal{G}_K}, r=r_\chi}\lim_{s \rightarrow 0}s^{-r}L_{k,S,T}(s,\chi^{-1})e_{\chi} \in \bbC[\mathcal{G}_K],$$
where $e_\chi:=|\mathcal{G}_K|^{-1}\sum_{\sigma \in \mathcal{G}_K}\chi(\sigma)\sigma^{-1}$. It is easy to see that $\theta_{K/k,S,T}^{(r)} \in \bbR[\mathcal{G}_K]$. Note that, when $r=0$, this is the usual Stickelberger element. Define 
$$X_{K,S}:=\{ \sum_{w\in S_K}a_ww \in \bigoplus_{w\in S_K}\bbZ w \ | \ \sum_{w \in S_K}a_w=0 \}.$$
Note that $X_{K,S}$ has a natural structure of $\mathcal{G}_K$-module, since $\mathcal{G}_K$ acts on $S_K$. We define 
$$\lambda_{K,S} : \mathcal{O}_{K,S,T}^\times \rightarrow \bbR\otimes_\bbZ X_{K,S}$$
by $\lambda_{K,S}(a):=-\sum_{w \in S_K}\log|a|_ww$, where $|\cdot|_w$ is the normalized absolute value at $w$. By Dirichlet's unit theorem, $\lambda_{K,S}$ induces an isomorphism of $\bbR[\mathcal{G}_K]$-modules
$$\bbR\otimes_\bbZ \mathcal{O}_{K,S,T}^\times \stackrel{\sim}{\rightarrow} \bbR\otimes_\bbZ X_{K,S}.$$ 

\subsection{The statement of the Rubin-Stark conjecture} \label{stateRS}
In this subsection, we state the Rubin-Stark conjecture. We need the following definition, due to Rubin (\cite[\S 1.2]{R2}). 

\begin{definition}
For $(K,S,T,V) \in \Omega(k)$, define
\begin{eqnarray}
\bigcap_{\mathcal{G}_K}^r \mathcal{O}_{K,S,T}^\times := \{ a \in \bbQ \otimes_\bbZ \bigwedge_{\mathcal{G}_K}^r\mathcal{O}_{K,S,T}^\times \ | \ \Phi(a)\in\bbZ[\mathcal{G}_K] \mbox{ for all }\Phi\in\bigwedge_{\mathcal{G}_K}^r\Hom_{\mathcal{G}_K}(\mathcal{O}_{K,S,T}^\times, \bbZ[\mathcal{G}_K]) \},\nonumber
\end{eqnarray}
where $r=r_V:=|V|$. (Note that $\bigcap_{\mathcal{G}_K}^0\mathcal{O}_{K,S,T}^\times=\bbZ[\mathcal{G}_K]$.)

\end{definition}

Note that $\bigcap$ is not the intersection.

{\it{From now, we fix a total order on the set of all places of $k$, and any exterior powers indexed by a set of places of $k$ is arranged by this fixed order. We also fix, for each place $v$ of $k$, a place $w$ of $k^{\rm sep}$ lying above $v$. For any finite separable extension $K/k$, the fixed place lying above $v$ is also denoted by $w$.}} 

\begin{definition}
Let $(K,S,T,V) \in \Omega(k)$, and put $r:=|V|$. Choose $v_0 \in S\setminus V$, and define 
$$x_{K,S,T,V}:=\theta_{K/k,S,T}^{(r)}\bigwedge_{v \in V}(w-w_0) \in \bbR \otimes_\bbZ \bigwedge_{\mathcal{G}_K}^r X_{K,S}.$$
\end{definition}

The following proposition shows that the element $x_{K,S,T,V}$ is well-defined, i.e. $x_{K,S,T,V}$ does not depend on the choice of $v_0\in S\setminus V$. 

\begin{proposition}
Let $(K,S,T,V) \in \Omega(k)$, and put $r:=|V|$. Take $v_0, v_0' \in S \setminus V$. Then we have 
$$\theta_{K/k,S,T}^{(r)}\bigwedge_{v \in V}(w-w_0) = \theta_{K/k,S,T}^{(r)}\bigwedge_{v\in V}(w-w_0') \quad in \quad \bbR \otimes_\bbZ \bigwedge_{\mathcal{G}_K}^r X_{K,S}.$$
\end{proposition}

\begin{proof}
If $r < \min\{|S|-1,|\{v\in S \ | \ v \mbox{ splits completely in }K  \}|\}$, then $\theta_{K/k,S,T}^{(r)}=0$, so the proposition is trivial. If $r=|S|-1$, then we must have $v_0=v_0'$, so there is nothing to prove. Hence we may assume $V=\{  v\in S \ | \ v \mbox{ splits completely in }K  \}$ and $r<|S|-1$. In this case, $v_0$ and $v_0'$ do not split completely in $K$, so we see that $e_\chi(w_0-w_0')=0$ (in $\bbC \otimes_\bbZ X_{K,S}$) for every $\chi \in \widehat{\mathcal{G}_K}$ such that $r_\chi=r$. The proposition follows by noting that $w-w_0'=(w-w_0)+(w_0-w_0')$.
\end{proof}

For any $r\in \bbZ_{\geq 0}$, the isomorphism 
$$\bbR \otimes_\bbZ \bigwedge_{\mathcal{G}_K}^r \mathcal{O}_{K,S,T}^\times \stackrel{\sim}{\rightarrow} \bbR \otimes_\bbZ \bigwedge_{\mathcal{G}_K}^r X_{K,S}$$
induced by $\lambda_{K,S}$ is also denoted by $\lambda_{K,S}$. 

Now we state the Rubin-Stark conjecture. 

\begin{conjecture}[The Rubin-Stark conjecture, {\cite[Conjecture B]{R2}}]
For $(K,S,T,V) \in \Omega(k)$, there exists a unique $\varepsilon_{K,S,T,V} \in \bigcap_{\mathcal{G}_K}^r \mathcal{O}_{K,S,T}^\times$ such that 
$$\lambda_{K,S}(\varepsilon_{K,S,T,V})=x_{K,S,T,V},$$
where $r=|V|$.
\end{conjecture}

\begin{remark}
Our formulation of the Rubin-Stark conjecture is slightly different from the original formulation of  Rubin in \cite[Conjecture B]{R2}. But from \cite[Proof of Proposition 2.4]{R2}, one easily sees that our formulation is equivalent to the original one. Note also that the unique element $\varepsilon_{K,S,T,V}$ predicted by this conjecture coincides with the one predicted by \cite[Conjecture B$'$]{R2}.
\end{remark}

The element $\varepsilon_{K,S,T,V}$ predicted by the Rubin-Stark conjecture is called {\it Rubin-Stark element}. 

\begin{remark} \label{RSknown}
The Rubin-Stark conjecture for $(K,S,T,V) \in \Omega(k)$ is known to be true, for example, in the following cases:
\begin{itemize}
\item[(i)]{$V=\emptyset$ (\cite[Theorem 3.3]{R2}),}
\item[(ii)]{$K$ is a finite abelian extension of $\bbQ$ or a function field (\cite[Theorem A]{B}),}
\item[(iii)]{all $v\in S$ split completely in $K$ (\cite[Proposition 3.1]{R2}).}
\end{itemize}
\end{remark}

\subsection{Some properties of Rubin-Stark elements} \label{properties}
In this subsection, we fix a finite set $T$ of finite places of $k$ such that $\Omega(k,T)\neq \emptyset$, and {\it assume that the Rubin-Stark conjecture holds for every $(K,S,T,V)$ such that $(K,S,V)\in \Omega(k,T)$}. For the proof of the following two propositions, see \cite{R2} or \cite{sano}. 

\begin{proposition}[{\cite[Proposition 6.1]{R2}, \cite[Proposition 3.5]{sano}}] \label{nr}
Let $(K,S,V),(K',S',V) \in \Omega(k,T)$, and suppose that $K\subset K'$ and $S\subset S'$. Then we have 
$$ \N_{K'/K}^r(\varepsilon_{K',S',T,V})=(\prod_{v \in S'\setminus S}(1-{\rm Fr}_v^{-1}))\varepsilon_{K,S,T,V} \quad \text{in} \quad \bigcap_{\mathcal{G}_K}^r\mathcal{O}_{K,S,T}^\times,$$
where $r= |V|$, $\N_{K'/K}:=\N_{\Gal(K'/K)}=\sum_{\sigma \in \Gal(K'/K)}\sigma$ and $\N_{K'/K}^r$ denotes the $r$-th power of $\N_{K'/K}$. (When $r=0$, $\N_{K'/K}^0$ means the natural map $\bbZ[\mathcal{G}_{K'}]\rightarrow \bbZ[\mathcal{G}_{K}]$.)
\end{proposition}

\begin{proposition}[{\cite[Proposition 5.2]{R2}, \cite[Proposition 3.6]{sano}}] \label{ordr}
Let $(K,S,V), (K,S',V') \in \Omega(k,T)$, and suppose that $S \subset S'$, $V\subset V'$ and $S'\setminus S=V'\setminus V$. Put 
$$\Phi_{V',V}=\sgn(V',V)\bigwedge_{v\in V'\setminus V}(\sum_{\sigma\in \mathcal{G}_K}\ord_w(\sigma(\cdot))\sigma^{-1})\in \bigwedge_{\mathcal{G}_K}^{r'-r}\Hom_{\mathcal{G}_K}(\mathcal{O}_{K,S',T}^\times, \bbZ[\mathcal{G}_K]),$$
where $r =|V|$, $r'=|V'|$, and $\sgn(V',V)=\pm 1$ is the sign of the permutation 
$$(V'\setminus V \ V) \mapsto V'.$$
Then we have 
$$\Phi_{V',V}(\varepsilon_{K,S',T,V'})=\varepsilon_{K,S,T,V} \quad \text{in} \quad \bigcap_{\mathcal{G}_K}^r\mathcal{O}_{K,S,T}^\times.$$
\end{proposition}

\section{The refined conjecture} \label{secconj}
In this section, we recall the formulation of \cite[Conjecture 3]{sano}. The main result of this paper is stated in \S \ref{mainresult} (Theorem \ref{mainthm}). {\it Throughout this section, we assume that the Rubin-Stark conjecture holds for every $(K,S,T,V)\in \Omega(k)$. In particular, note that Conjecture \ref{sanoconj} and Theorem \ref{mainthm} are stated under the assumption that the Rubin-Stark conjecture holds for every $(K,S,T,V)\in \Omega(k)$.}

\subsection{The statement of the conjecture}
Let $S$ and $T$ be finite sets of places of $k$. Let $\Upsilon(k,S,T)$ be the set of quadruples $(K,L,V,V')$ satisfying the following:
\begin{itemize}
\item{$(K,S,T,V), (L,S,T,V')\in \Omega(k)$,}
\item{$L \subset K$,}
\item{$V \subset V'$.}
\end{itemize}
Assume $\Upsilon(k,S,T)\neq \emptyset$, and fix $(K,L,V,V') \in \Upsilon(k,S,T)$. We use the following notations:
\begin{itemize}
\item{$r:=|V|$,}
\item{$r':=|V'|$,}
\item{$d:=r'-r$,}
\item{$G:=\mathcal{G}_K(=\Gal(K/k))$,}
\item{$H:=\Gal(K/L)$,}
\item{$I(H):=\ker(\bbZ[H]\rightarrow \bbZ)$ (the augmentation ideal),}
\item{$I_H:=I(H)\bbZ[G](=\ker(\bbZ[G] \rightarrow \bbZ[G/H]))$.}
\end{itemize}
For $n \in \bbZ_{\geq 0},$
\begin{itemize}
\item{$Q(H)^n:=I(H)^n/I(H)^{n+1}$,}
\item{$Q_H^n:=I_H^n/I_H^{n+1}$.}
\end{itemize}
It is easy to see that there is a natural isomorphism of $G/H$-modules 
$$\bbZ[G/H]\otimes_\bbZ Q(H)^n \simeq Q_H^n.$$
We often identify these $G/H$-modules.

We define 
$$\mathcal{N}_{K/L} : \bigcap_G^r\mathcal{O}_{K,S,T}^\times \rightarrow (\bigcap_G^r\mathcal{O}_{K,S,T}^\times )\otimes_\bbZ \bbZ[H]/I(H)^{d+1}$$
by 
$$\mathcal{N}_{K/L}(a) := \sum_{\sigma \in H} \sigma a \otimes \sigma^{-1}.$$

For $v \in V' \setminus V$, define 
$$\varphi_v=\varphi_v^{K/L} : \mathcal{O}_{L,S,T}^\times \rightarrow Q_H^1$$
by $\varphi_v(a):=\sum_{\sigma \in G/H}({\rm rec}_w(\sigma a)-1) \sigma^{-1}$, where ${\rm rec}_w$ is the reciprocity map at $w$. By \cite[Proposition 2.7]{sano}, $\bigwedge_{v\in V'\setminus V}\varphi_v \in \bigwedge_{G/H}^d\Hom_{G/H}(\mathcal{O}_{L,S,T}^\times,Q_H^1)$ defines the homomorphism (``algebraic regulator map")
$$\bigwedge_{v\in V'\setminus V}\varphi_v : \bigcap_{G/H}^{r'}\mathcal{O}_{L,S,T}^\times \rightarrow (\bigcap_{G/H}^{r}\mathcal{O}_{L,S,T}^\times) \otimes_{\bbZ}Q(H)^d.$$

Recall the definition of the canonical injection
$$\nu_{K/L}:\bigcap_{G/H}^r\mathcal{O}_{L,S,T}^\times \rightarrow \bigcap_G^r \mathcal{O}_{K,S,T}^\times$$
constructed in \cite[Lemma 2.11]{sano}. Define 
$$\iota_G : \bigwedge_G^r \Hom_G(\mathcal{O}_{K,S,T}^\times,\bbZ[G]) \rightarrow \Hom_G(\bigwedge_G^r \mathcal{O}_{K,S,T}^\times,\bbZ[G])$$
by $\iota_G(\varphi_1\wedge\cdots\wedge\varphi_r)(u_1\wedge\cdots\wedge u_r)=\det(\varphi_i(u_j))_{1\leq i,j\leq r}$. (This is the map constructed in (\ref{extmap}).) It is not difficult to see that the map
$$\alpha_G:\bigcap_G^r\mathcal{O}_{K,S,T}^\times \rightarrow \Hom_G(\im \iota_G,\bbZ[G])$$
defined by $\alpha_G(a)(\Phi)=\Phi(a)$ is an isomorphism (see \cite[\S 1.2]{R2}). Similarly we can define the map 
$$\iota_{G/H} : \bigwedge_{G/H}^r \Hom_{G/H}(\mathcal{O}_{L,S,T}^\times,\bbZ[G/H]) \rightarrow \Hom_{G/H}(\bigwedge_{G/H}^r \mathcal{O}_{L,S,T}^\times, \bbZ[G/H]),$$
and we have the isomorphism 
$$\alpha_{G/H}:\bigcap_{G/H}^r\mathcal{O}_{L,S,T}^\times \stackrel{\sim}{\rightarrow} \Hom_{G/H}(\im \iota_{G/H},\bbZ[G/H]).$$
Let $\kappa : \bbZ[G/H] \stackrel{\sim}{\rightarrow} \bbZ[G]^H$ be the isomorphism defined by $1\mapsto \N_H$. Define 
$$\beta_{K/L}:  \Hom_{G/H}(\im \iota_{G/H},\bbZ[G/H]) \rightarrow \Hom_{G}(\im \iota_{G},\bbZ[G])$$
by $\beta_{K/L}(f)(\Phi)=\kappa(f(\Phi^H))$, where $\Phi^H \in \im \iota_{G/H}$ is the image of $\Phi \in\im \iota_G$ under the map $\im \iota_G\rightarrow \im\iota_{G/H}$ induced by the map 
$$\bigwedge_G^r \Hom_G(\mathcal{O}_{K,S,T}^\times,\bbZ[G]) \rightarrow \bigwedge_{G/H}^r\Hom_{G/H}(\mathcal{O}_{L,S,T}^\times,\bbZ[G/H])$$
defined by $\varphi_1\wedge\cdots\wedge\varphi_r \mapsto (\kappa^{-1}\circ\varphi_1)\wedge\cdots\wedge(\kappa^{-1}\circ\varphi_r)$. 
Now we define 
$$\nu_{K/L}:=\alpha_G^{-1}\circ \beta_{K/L} \circ \alpha_{G/H}.$$
Note that, if $r=0$, then we have $\nu_{K/L}=\kappa$. 
As proved in \cite[Lemma 2.11]{sano},  the map $\nu_{K/L}$ is injective. The same result shows that the map 
$$(\bigcap_{G/H}^{r}\mathcal{O}_{L,S,T}^\times) \otimes_{\bbZ}Q(H)^d \rightarrow( \bigcap_G^r\mathcal{O}_{K,S,T}^\times) \otimes_\bbZ Q(H)^d \rightarrow (\bigcap_G^r\mathcal{O}_{K,S,T}^\times )\otimes_\bbZ \bbZ[H]/I(H)^{d+1}$$
induced by $\nu_{K/L}$ and the inclusion $Q(H)^d \hookrightarrow \bbZ[H]/I(H)^{d+1}$ is also injective. This injection is also denoted by $\nu_{K/L}$. 

\begin{conjecture}[{\cite[Conjecture 3]{sano}, \cite[Conjecture 5.2]{MR}}] \label{sanoconj}
We have 
$$\mathcal{N}_{K/L}(\varepsilon_{K,S,T,V})\in  \im \nu_{K/L},$$
and an equality
$$\nu_{K/L}^{-1}(\mathcal{N}_{K/L}(\varepsilon_{K,S,T,V}))=\sgn(V',V)(\bigwedge_{v\in V'\setminus V}\varphi_v)(\varepsilon_{L,S,T,V'}).$$
($\sgn(V',V)$ is as in Proposition \ref{ordr}.)
\end{conjecture}

\begin{remark}
In the case $r=0$, the definitions of the maps corresponding to $\mathcal{N}_{K/L}$ and $\nu_{K/L}$  are different in the original conjecture \cite[Conjecture 3]{sano}. More precisely, in the case $r=0$, define 
$$\mathcal{N}_{K/L}' : \bigcap_G^0 \mathcal{O}_{K,S,T}^\times =\bbZ[G] \rightarrow \bbZ[G]/I_H^{d+1}$$
to be the natural map, and 
$$\nu_{K/L}' : (\bigcap_{G/H}^0\mathcal{O}_{L,S,T}^\times) \otimes_\bbZ Q(H)^d =\bbZ[G/H]\otimes_\bbZ Q(H)^d\simeq Q_H^d \hookrightarrow \bbZ[G]/I_H^{d+1}$$
to be the natural injection. Then \cite[Conjecture 3]{sano} in this case claims 
$$\mathcal{N}_{K/L}'(\varepsilon_{K,S,T,\emptyset})  \in  \im \nu_{K/L}' (=Q_H^d),$$
and an equality
$$\nu_{K/L}'^{-1}(\mathcal{N}_{K/L}'(\varepsilon_{K,S,T,\emptyset}))=(\bigwedge_{v\in V'}\varphi_v)(\varepsilon_{L,S,T,V'}).$$
In \cite[Lemma 5.6]{MR}, Mazur and Rubin observed that $\mathcal{N}_{K/L}'(\varepsilon_{K,S,T,\emptyset})\in  \im \nu_{K/L}'$ if and only if $\mathcal{N}_{K/L}(\varepsilon_{K,S,T,\emptyset})\in  \im \nu_{K/L}(=\bbZ[G]^H \otimes_\bbZ Q(H)^d)$, and if this equivalent conditions are satisfied, then 
$$\nu_{K/L}^{-1}(\mathcal{N}_{K/L}(\varepsilon_{K,S,T,\emptyset}))=\nu_{K/L}'^{-1}(\mathcal{N}_{K/L}'(\varepsilon_{K,S,T,\emptyset})).$$
Hence, the formulation of Conjecture \ref{sanoconj} is equivalent to \cite[Conjecture 3]{sano}. Note also that, the injection $\nu_{K/L}$ is essentially the same as $\bold{j}_{K/L}$ defined in \cite[Lemma 4.9]{MR} (note that our $K/L$ is $L/K$ in \cite{MR}), so Conjecture \ref{sanoconj} is equivalent to the conjecture of Mazur and Rubin in \cite[Conjecture 5.2]{MR}.

\end{remark}

\begin{remark} \label{sanotrue}
The result of Burns, Kurihara, and the author in \cite{BKS} shows that Conjecture \ref{sanoconj} is true in the case that $k=\bbQ$ or $k$ is a function field. 
\end{remark}



For later use, we record some properties of the injection $\nu_{K/L}$. 

\begin{lemma} \label{lemnu}
\begin{itemize}
\item[(i)]{For every $a\in \bigcap_G^r \mathcal{O}_{K,S,T}^\times$, we have 
$$\nu_{K/L}(\N_{K/L}^r(a))=\N_{K/L} a,$$
where $\N_{K/L}:=\N_H$. (When $r=0$, $\N_{K/L}^0$ means the natural map $\bbZ[G]\rightarrow \bbZ[G/H]$.)}
\item[(ii)]{For any intermediate field $K'$ of $K/L$, we have 
$$\nu_{K/L}=\nu_{K/K'}\circ\nu_{K'/L}\quad \text{on} \quad \bigcap_{G/H}^r\mathcal{O}_{L,S,T}^\times.$$}
\end{itemize}
\end{lemma}
\begin{proof}
This is easy, so we omit the proof. 
\end{proof}

Conjecture \ref{sanoconj} has a natural functorial property as follows.

\begin{proposition} \label{functorial}
Assume that Conjecture \ref{sanoconj} holds for $(K,L,V,V')\in \Upsilon(k,S,T)$. Then Conjecture \ref{sanoconj} holds for $(K',L,V,V')\in \Upsilon(k,S,T)$ such that $K' \subset K$.
\end{proposition}

\begin{proof}
Set $G=\Gal(K/k)$, $H=\Gal(K/L)$, $G'=\Gal(K'/k)$, and $H'=\Gal(K'/L)$. Let $\pi$ denote the restriction map $H \rightarrow H'$. 
The map 
$$(\bigcap_{G}^{r}\mathcal{O}_{K,S,T}^\times) \otimes_{\bbZ}\bbZ[H]/I(H)^{d+1} \rightarrow (\bigcap_{G}^{r}\mathcal{O}_{K,S,T}^\times) \otimes_{\bbZ}\bbZ[H']/I(H')^{d+1}$$
induced by $\pi$ is also denoted by $\pi$. For each $\sigma\in H'$, fix a lift $\widetilde \sigma \in H$. Then we compute
\begin{eqnarray}
\pi(\mathcal{N}_{K/L}(\varepsilon_{K,S,T,V}))&=&\sum_{\sigma \in H'}\widetilde \sigma(\N_{K/K'}\varepsilon_{K,S,T,V})\otimes \sigma^{-1} \nonumber \\
&=&\sum_{\sigma\in H'}\widetilde\sigma(\nu_{K/K'}(\N_{K/K'}^r(\varepsilon_{K,S,T,V}))) \otimes \sigma^{-1}\nonumber \\
&=&\sum_{\sigma\in H'}\widetilde\sigma(\nu_{K/K'}(\varepsilon_{K',S,T,V})) \otimes \sigma^{-1}\nonumber \\
&=&\nu_{K/K'}(\mathcal{N}_{K'/L}(\varepsilon_{K',S,T,V})), \nonumber
\end{eqnarray}
where the first equality follows from direct computation, the second from Lemma \ref{lemnu} (i), and the third from Proposition \ref{nr}. 
By the functoriality of reciprocity maps, we have
$$\pi((\bigwedge_{v\in V'\setminus V}\varphi_v^{K/L})(\varepsilon_{L,S,T,V'}))=(\bigwedge_{v\in V'\setminus V}\varphi_v^{K'/L})(\varepsilon_{L,S,T,V'}).$$
Hence, assuming Conjecture \ref{sanoconj} for $(K,L,V,V')$, we have 
$$\nu_{K/K'}(\mathcal{N}_{K'/L}(\varepsilon_{K',S,T,V}))=\sgn(V',V)\nu_{K/L}((\bigwedge_{v\in V'\setminus V}\varphi_v^{K'/L})(\varepsilon_{L,S,T,V'})).$$
Since $\nu_{K/L}=\nu_{K/K'}\circ \nu_{K'/L}$ by Lemma \ref{lemnu} (ii), the proposition follows from the injectivity of $\nu_{K/K'}$ and $\nu_{K'/L}$.
\end{proof}

\subsection{The statement of the main theorem} \label{mainresult}

\begin{theorem} \label{mainthm}
Assume:
\begin{itemize}
\item[(i)]{$S_\infty(k) \subset V$,}
\item[(ii)]{all $v \in S$ split completely in $L$,}
\item[(iii)]{$H = \prod_{v \in S \setminus V}J_v,$}
\end{itemize}
where $J_v \subset G$ is the inertia group at $v$. Then Conjecture \ref{sanoconj} is true. 
\end{theorem}

\begin{example}
The assumptions in Theorem \ref{mainthm} are satisfied in the following case. Let $L$ be the Hilbert class field of $k$. Take principal prime ideals $\frak p_1,\ldots,\frak p_n$, and put $S:=S_\infty(k)\cup \{ \frak p_1,\ldots,\frak p_n\}$. Let $K$ be the composite field of the ray class fields modulo $\frak p_i^{e_i}$'s, where $e_i$ is a positive integer. If we set $V:=S_\infty(k)$, then the assumptions (i)-(iii) are satisfied.
\end{example}

\begin{remark}
To prove Theorem \ref{mainthm}, we do not need to assume that the Rubin-Stark conjecture holds for {\it every} $(K,S,T,V)\in \Omega(k)$. More precisely, Remark \ref{RSknown} (iii) and the proof which we will describe in the next section show that we only need to assume that the Rubin-Stark conjecture holds for $(K_X,S_X,T,V)$ for every nonempty subset $X \subset S\setminus V$, where $K_X$ is the unique intermediate field of $K/L$ such that $\Gal(K_X/L)=\prod_{v\in X}J_v$, and $S_X=V\cup X$.
\end{remark}

By Proposition \ref{functorial}, we have the following corollary. 

\begin{corollary}
Assume the assumptions of Theorem \ref{mainthm} hold for $(K,L,V,V')\in\Upsilon(k,S,T)$. Then Conjecture \ref{sanoconj} is true for $(K',L,V,V')\in \Upsilon(k,S,T)$ such that $K'\subset K$.
\end{corollary}

\section{Proof} \label{secpr}
In this section, we give a proof of Theorem \ref{mainthm}. 


We assume that the assumptions (i)-(iii) in Theorem \ref{mainthm} are satisfied. By the assumption (ii), note that Theorem \ref{mainthm} is reduced to the case that $r'=|S|-1$, by \cite[Proposition 3.12]{sano}. Henceforth we assume that $V'=S\setminus \{v_0\}$ with some $v_0\in S \setminus V$. 

\begin{lemma} \label{lem1}
For any $v_0' \in S\setminus V$, we have 
$$\sgn(V',V)(\bigwedge_{v\in V'\setminus V}\varphi_v)(\varepsilon_{L,S,T,V'})=\sgn(V'',V)(\bigwedge_{v\in V''\setminus V}\varphi_v)(\varepsilon_{L,S,T,V''}),$$
where $V''=S\setminus \{v_0'\}$.
\end{lemma}

\begin{proof}
By the product formula of reciprocity maps, we see that 
\begin{eqnarray}
\sum_{v\in S\setminus V}\varphi_v^{K/L}=0 \quad \text{on} \quad \mathcal{O}_{k,S,T}^\times. \label{eq2}
\end{eqnarray}
Since all $v\in S$ split completely in $L$, we see that $\varepsilon_{L,S,T,V'},\varepsilon_{L,S,T,V''} \in e_{1}(\bbQ \otimes_\bbZ \bigwedge_{G/H}^{r'}\mathcal{O}_{L,S,T}^\times)=\bbQ\otimes_\bbZ \bigwedge_\bbZ^{r'} \mathcal{O}_{k,S,T}^\times$. We also see that $\varepsilon_{L,S,T,V'}=\pm \varepsilon_{L,S,T,V''}$ by the characterization of Rubin-Stark elements. Hence, by (\ref{eq2}), we have 
$$(\bigwedge_{v\in V'\setminus V}\varphi_v)(\varepsilon_{L,S,T,V'})=\pm(\bigwedge_{v\in V''\setminus V}\varphi_v)(\varepsilon_{L,S,T,V''}).$$
The lemma follows from explicit computation of sign.
\end{proof}

\begin{remark}
The proof of \cite[Proposition 3.1]{R2} shows that the Rubin-Stark element $\varepsilon_{L,S,T,V'}$ is described explicitly as follows:
$$\varepsilon_{L,S,T,V'}= \frac{|A_{k,S,T}|}{|G/H|^{r'}}u_1\wedge \cdots \wedge u_{r'},$$
where $A_{k,S,T}$ is the ``$S$-ray class group modulo $T$" (see \cite[\S 1.1]{R2}), and $\{u_i\}$ is a basis of $\mathcal{O}_{k,S,T}^\times$ such that 
$$(\bigwedge_{v\in S\setminus \{v_0\}}(-\log|\cdot|_v))(u_1\wedge\cdots\wedge u_{r'})<0.$$
Lemma \ref{lem1} can also be proved by using this description. 
\end{remark}

We set some notations. Put $W:=S \setminus V$. For each subset $X\subset W$, define 
$$H_X:=\prod_{v\in X}J_v.$$
$H_X$ is regarded as a quotient of $H$, and also a subgroup of $H$. Let $K_X$ denote the unique intermediate field of $K/L$ such that $\Gal(K_X/L)=H_X$. Put $G_X:=\Gal(K_X/k)$ and $S_X:=V\cup X$. Note that $G_W=G$, $H_W=H$, $K_W=K$, and $S_W=S$. Define a map 
$$\pi_X : H \rightarrow H_X \hookrightarrow H,$$
where the first arrow is the natural projection, and the second is the natural inclusion. The endomorphism of $\bbZ[H]$ induced by $\pi_X$ is also denoted by $\pi_X$. If $X \neq \emptyset$, choose $v_0' \in X$, then we easily see that $(K_X,L,V,V_X') \in \Upsilon(k,S_X,T)$, where $V_X':=S_X\setminus \{v_0'\}$. We define 
$$L_X:=\mathcal{N}_{K_X/L}(\varepsilon_{K_X,S_X,T,V}) \in (\bigcap_{G_X}^r \mathcal{O}_{K_X,S_X,T}^\times )\otimes_\bbZ \bbZ[H_X],$$
$$R_X:=\sgn(V_X',V)(\bigwedge_{v \in V_X'\setminus V}\varphi_v)(\varepsilon_{L,S_X,T,V_X'}) \in (\bigcap_{G/H}^r \mathcal{O}_{L,S_X,T}^\times) \otimes_\bbZ Q(H_X)^{|X|-1}.$$
Note that, by Lemma \ref{lem1}, $R_X$ does not depend on the choice of $v_0' \in X$. 

In the next lemma, the endomorphisms of $(\bigcap_{G}^r \mathcal{O}_{K,S,T}^\times) \otimes_\bbZ \bbZ[H]/I(H)^{d+1}$ and $(\bigcap_{G/H}^r \mathcal{O}_{L,S,T}^\times) \otimes_\bbZ Q(H)^d$ induced by $\pi_X$ are also denoted by $\pi_X$.

\begin{lemma} \label{lem2}
Let $X \subset W$ be a nonempty subset. Then:
\begin{itemize}
\item[(i)]{$$\pi_X(L_W)=\nu_{K/K_X}(L_X) \cdot(1\otimes \prod_{v \in W\setminus X}(1-{\rm Fr}_v^{-1})) \ \ \text{in} \ \ (\bigcap_{G}^r \mathcal{O}_{K,S,T}^\times) \otimes_\bbZ \bbZ[H]/I(H)^{d+1},$$}
\item[(ii)]{$$\pi_X(R_W)=R_X \cdot(1\otimes \prod_{v \in W\setminus X}({\rm Fr}_v-1) )\quad \text{in} \quad (\bigcap_{G/H}^r \mathcal{O}_{L,S,T}^\times) \otimes_\bbZ Q(H)^d.$$}
\end{itemize}
(Here ${\rm Fr}_v$ is considered to be in $H_X$, hence in $H$.)
\end{lemma}

\begin{proof}
For each $\sigma \in H_X$, fix a lift $\widetilde\sigma\in H$. 
We compute 
\begin{eqnarray}
\pi_X(L_W)
&=& \sum_{\sigma \in H_X}\widetilde \sigma(\N_{K/K_X}\varepsilon_{K,S,T,V}) \otimes \sigma^{-1}\nonumber \\
&=& \sum_{\sigma\in H_X}\widetilde \sigma(\nu_{K/K_X}(\N_{K/K_X}^r(\varepsilon_{K,S,T,V})))\otimes\sigma^{-1} \nonumber \\
&=& \nu_{K/K_X}(\mathcal{N}_{K_X/L}((\prod_{v \in W\setminus X}(1-{\rm Fr}_v^{-1}))\varepsilon_{K_X,S_X,T,V})) \nonumber \\
&=& \nu_{K/K_X}(\mathcal{N}_{K_X/L}(\varepsilon_{K_X,S_X,T,V} )) \cdot(1\otimes \prod_{v\in W \setminus X}(1-{\rm Fr}_v^{-1}) )\nonumber \\
&=& \nu_{K/K_X}(L_X) \cdot(1\otimes \prod_{v \in W\setminus X}(1-{\rm Fr}_v^{-1})), \nonumber
\end{eqnarray}
where the second equality follows from Lemma \ref{lemnu} (i), the third from Proposition \ref{nr}, and the fourth from direct computation. This shows (i). 

Next, we compute $\pi_X(R_W)$. By Lemma \ref{lem1}, we may assume $v_0 \in X$ and $V_X'=S_X\setminus \{ v_0 \}$. Note that, for $v\in W\setminus X$, we have 
\begin{eqnarray}
\varphi_v^{K_X/L}=\sum_{\sigma \in G/H}\ord_w(\sigma(\cdot))\sigma^{-1}({\rm Fr}_v-1), \label{eq}
\end{eqnarray}
since $v$ is unramified in $K_X$ (see \cite[Proposition 13, Chpt. XIII]{Se}). We compute 
\begin{eqnarray}
\pi_X(R_W) &=& \sgn(V',V)(\bigwedge_{v\in W\setminus \{ v_0\}}\varphi_v^{K_X/L})(\varepsilon_{L,S,T,V'}) \nonumber \\
&=& \sgn(V',V)\sgn(W\setminus \{v_0\},V_X'\setminus V)(\bigwedge_{v\in V_X'\setminus V }\varphi_v^{K_X/L}) \nonumber \\
&&\circ(\bigwedge_{v\in W\setminus X}(\sum_{\sigma \in G/H}\ord_w(\sigma(\cdot))\sigma^{-1}))(\varepsilon_{L,S,T,V'})\cdot(1\otimes \prod_{v \in W\setminus X}({\rm Fr}_v-1)) \nonumber \\
&=&\sgn(V',V)\sgn(W\setminus \{v_0\},V_X'\setminus V)\sgn(V',V_X') \nonumber \\
&& \times (\bigwedge_{v\in V_X'\setminus V}\varphi_v^{K_X/L})(\varepsilon_{L,S_X,T,V_X'}) \cdot(1\otimes\prod_{v \in W\setminus X}({\rm Fr}_v-1) )\nonumber \\
&=& \sgn(V_X',V)(\bigwedge_{v \in V_X'\setminus V}\varphi_v^{K_X/L})(\varepsilon_{L,S_X,T,V_X'})\cdot(1\otimes \prod_{v \in W\setminus X}({\rm Fr}_v-1)) \nonumber \\
&=& R_X \cdot(1\otimes \prod_{v \in W\setminus X}({\rm Fr}_v-1)), \nonumber 
\end{eqnarray}
where the second equality follows from (\ref{eq}), the third equality from Proposition \ref{ordr}, and the fourth from sign computation. This shows (ii).
\end{proof}

\begin{lemma} \label{lem3}
$$\pi_\emptyset(L_W)=\pi_\emptyset(\nu_{K/L}(R_W)).$$
\end{lemma}

\begin{proof}
This follows from Proposition \ref{nr} and Lemma \ref{lemnu} (i). 
\end{proof}

The following algebraic lemma is due to Darmon's method (\cite[\S 8]{D}).

\begin{lemma} \label{lem4}
Let $a \in \bigcap_{G}^r \mathcal{O}_{K,S,T}^\times \otimes_\bbZ \bbZ[H]/I(H)^{d+1}$. Then we have 
$$a=-\sum_{X \subset W,X\neq W}(-1)^{|W\setminus X|}\pi_X(a).$$
\end{lemma}

\begin{proof}
Take $\sigma \in H$, and write $\sigma = \prod_{v\in W}\sigma_v$ with $\sigma_v \in J_v$. Then we have 
$$\sum_{X \subset W}(-1)^{|W\setminus X|}\pi_X(\sigma)=\prod_{v \in W}(\sigma_v-1) \in I(H)^{|W|}=I(H)^{d+1}.$$
From this, we see that 
$$\sum_{X\subset W}(-1)^{|W\setminus X|}\pi_X=0 \quad \text{on} \quad (\bigcap_{G}^r \mathcal{O}_{K,S,T}^\times) \otimes_\bbZ \bbZ[H]/I(H)^{d+1}.$$
Since $\pi_W=\id$, the lemma follows.
\end{proof}

\begin{proof}[Proof of Theorem \ref{mainthm}]
We prove that the equality $L_W=\nu_{K/L}(R_W) $ holds in $(\bigcap_G^r\mathcal{O}_{K,S,T}^\times)\otimes_\bbZ \bbZ[H]/I(H)^{d+1}$ by induction on $|W|$. When $|W|=1$, this follows from Proposition \ref{nr} and Lemma \ref{lemnu} (i). We assume $L_X=\nu_{K_X/L}(R_X)$ for all proper nonempty subsets $X \subset W$. By Lemma \ref{lem4}, it is sufficient to prove that $\pi_X(L_W)=\pi_X(\nu_{K/L}(R_W))$ for each proper subset $X \subset W$. If $X=\emptyset$, then this follows from Lemma \ref{lem3}. Suppose $X \neq \emptyset$. Note that, by the inductive hypothesis, we have $\nu_{K/K_X}(L_X)=\nu_{K/L}(R_X) \in (\bigcap_G^r\mathcal{O}_{K,S,T}^\times )\otimes_\bbZ Q(H)^{|X|-1}$. Hence, by Lemma \ref{lem2} (i), we have 
$$\pi_X(L_W)=\nu_{K/K_X}(L_X) \cdot(1\otimes \prod_{v \in W\setminus X}({\rm Fr}_v-1))=\nu_{K/L}(R_X)\cdot(1\otimes\prod_{v\in W \setminus X}({\rm Fr}_v-1)).$$
On the other hand, we know by Lemma \ref{lem2} (ii) that 
$$\pi_X(R_W)=R_X \cdot(1\otimes\prod_{v \in W\setminus X}({\rm Fr}_v-1)),$$
so we have $\pi_X(L_W)=\pi_X(\nu_{K/L}(R_W))$.
\end{proof}

\section*{Acknowledgement}
The author would like to thank Professor Masato Kurihara for his constant encouragement, and helpful advice. He also wishes to thank Professor David Burns for stimulating discussions on the subject of this paper, and providing him a good atmosphere when he visited King's College London.

\end{document}